\documentclass[11pt,a4paper,reqno]{article} 
\usepackage{authblk}
\usepackage[utf8]{inputenc} 

\usepackage{geometry} 
\usepackage{graphicx} 
\usepackage{graphics}
\usepackage{float}
\usepackage[usenames, dvipsnames]{color}

\usepackage{booktabs} 
\usepackage{array} 
\usepackage{paralist} 
\usepackage{verbatim} 
\usepackage{subfig} 
\usepackage{xcolor}
\usepackage{pstricks-add}
\usepackage{pstricks}
\usepackage{pst-plot} 
\usepackage{amsfonts}
\usepackage{amssymb}
\usepackage{amsthm}
\usepackage{mathrsfs}
\usepackage{amsmath}
\usepackage{bbm}
\usepackage{dsfont}
\usepackage{mathabx}
 \usepackage{color}

\usepackage{multirow}

\usepackage{fancyhdr} 
\usepackage{sectsty}
\allsectionsfont{\sffamily\mdseries\upshape} 

\usepackage[nottoc,notlof,notlot]{tocbibind} 
\usepackage[titles,subfigure]{tocloft} 

\newcommand{\Real}{{\mathbb{R}}}
\newcommand{\beq}{\begin{equation}}
\newcommand{\eeq}{\end{equation}}

\newcommand{\beqa}{\begin{eqnarray}}
\newcommand{\eeqa}{\end{eqnarray}}
\theoremstyle{plain}
\newtheorem{theorem}{Theorem}[section]

\newtheorem{lemma}[theorem]{Lemma}

\theoremstyle{remark}
\newtheorem{remark}[theorem]{Remark}

\numberwithin{equation}{section}

\def\ds{\displaystyle}
\def\ni{\noindent}
\def\bs{\bigskip}
\def\ms{\medskip}

\def\eps{\varepsilon}

\def\pa{\partial}

\def\RR{\mathbb{R}}
\def\ZZ{\mathbb{Z}}

\def\NN{\mathbb{N}}
\def\TT{\mathbb{T}}

\def\<{ \langle}
\def\>{ \rangle}
\DeclareMathOperator{\RE}{Re}
\DeclareMathOperator{\IM}{Im}

\def\un{{\mathbbmss{1}}}

\begin{document}
\title{Nonlinear instability of inhomogeneous steady states solutions to the HMF Model}
\author[$\dagger$]{M. Lemou}
\author[$\star$]{A. M. Luz}
\author[$\ddag$]{F. M\'ehats}

\affil[$\dagger$]{Univ Rennes, CNRS, IRMAR, {mohammed.lemou@univ-rennes1.fr}}
\affil[$\star$]{IME, Universidade Federal Fluminense, {analuz@id.uff.br}}
\affil[$\ddag$]{Univ Rennes, IRMAR, {florian.mehats@univ-rennes1.fr}}

\maketitle
\begin{abstract} 
In this work we prove the nonlinear instability of inhomogeneous steady states solutions to the Hamiltonian Mean Field (HMF) model. We first study the linear instability of this model under a simple criterion by adapting the techniques developed in \cite{LIN1}. In a second part, we extend to the inhomogeneous case some techniques developed in \cite{grenier, Han-Hauray, Han-Nguyen} and prove a nonlinear instability result under the same criterion.
\end{abstract}
\section{Introduction}

In this paper, we are interested in the nonlinear instability of inhomogeneous steady states  of the Hamiltonian Mean Field (HMF) system. The HMF system is a kinetic model describing particles moving on a unit circle interacting via an infinite range attractive cosine potential. This 1D model holds many qualitative properties of more realistic long-range interacting systems as the Vlasov-Poisson model. The HMF model has been the subject of many works in the physical community, for the study of non equilibrium phase transitions \cite{chavanis1,chavanis2,antoniazzi,ogawa2}, of travelling clusters \cite{yamaguchi1,yamaguchi2} or of relaxation processes \cite{barre1,barre2,chavanis4}. 
The long-time validity of the N-particle approximation for the HMF model has been investigated in \cite{rousset1,rousset2} and the Landau-damping phenomenon near a spatially homogeneous state has been studied recently in \cite{faou-rousset}. The formal linear stability of inhomogeneous steady states has been studied in \cite{chavanis3,OGW,yamaguchi3}. In particular, a simple criterion of linear stability has been derived in \cite{OGW}. In \cite{llmehats}, the authors of the present paper have proved that, under the same criterion $\kappa_0<1$ (see below for a precise formulation), the inhomogeneous steady states of HMF that are nonincreasing functions of the microscopic energy are nonlinearly stable. The aim of the present paper is to show, in a certain sense, that this criterion is sharp: we show that if $\kappa_0>1$, the HMF model can develop instabilities, from both the linear and the nonlinear points of view.

In \cite{LIN2}, Guo and Lin have derived a sufficient criterion for linear instability to 3D Vlasov-Poisson by extending an approach developped in \cite{LIN1} for BGK waves. Let us also mention that both works have adapted some techniques presented in \cite{guo-strauss} to prove the nonlinear instability of the 3D Vlasov-Poisson system. In the first part of this article we adapt these techniques and prove the linear instability of nonhomogeneous steady states to the HMF system. In \cite{Han-Hauray}, a nonlinear instability result for 1D Vlasov-Poisson equation was obtained for an initial data close to stationary homogeneous profiles that satisfy a Penrose instability criterion by using an approach developed in \cite{grenier}. In \cite{Han-Nguyen}, starting with the N- particles version of the HMF model,  a nonlinear instability result is obtained for the corresponding Vlasov approximation by also considering a Penrose instability condition for stationary homogeneous profiles. In the second part of this article our aim is to prove the nonlinear instability of non-homogeneous steady states  of HMF by adapting the techniques developed in \cite{grenier, Han-Hauray, Han-Nguyen}.

In the HMF model, the distribution function of particles $f(t,\theta,v)$ solves the initial-valued problem
\begin{align}\label{sis1}
&\partial_t f+v\partial_{\theta} f-\partial_{\theta} \phi_f\partial_v f=0,\qquad(t,\theta,v)\in\Real_{+}\times\TT\times\Real,\\
&f(0,\theta,v)=f_{\rm init}(\theta,v)\geq 0,\nonumber
\end{align}
where $\TT$ is the flat torus $[0,2\pi]$ and where the self-consistent potential $\phi_f$ associated with a distribution function $f$ is defined by
\begin{equation}
\label{phi}
 \phi_f(\theta)=-\int^{2 \pi}_{0}\rho_f(\theta')\cos(\theta-\theta')d\theta', \qquad \rho_f(\theta)=\int_{\Real}f(\theta,v)dv.
\end{equation}

Introducing the so-called magnetization vector defined by
\begin{equation}\label{mag2}
M_f=\int^{2 \pi}_{0}\rho_f(\theta)u(\theta) d\theta, \qquad \mbox{with}\quad u(\theta)=(\cos \theta,\sin \theta)^T
\end{equation}
we have
\begin{equation}\label{phimag}
 \phi_f(\theta)=-M_f\cdot u(\theta).
\end{equation}

In this work will consider steady states of \eqref{sis1} of the form
\begin{equation}\label{Qexp}
f_0(\theta,v)=F(e_0(\theta,v)),
\end{equation}
where F is a given nonnegative function and where the microscopic energy $e_0(\theta,v)$ is given by 
\begin{equation}\label{eqe0}
e_0(\theta,v)=\frac{v^2}{2}+\phi_0(\theta) \quad \text{with}\,\,\phi_0=\phi_{f_0}.
\end{equation}
Without loss of generality, we assume that $\phi_0(\theta)=-m_0\cos \theta$  with $m_0>0$. Here $m_0$ is the magnetization of the stationary state $f_0$ defined by $m_0=\int \rho_{f_0}\cos \theta d\theta$.

It is shown in \cite{llmehats} that (essentially) if $F$ is decreasing then $f_0$ is nonlinearly stable by the  HMF flow \eqref{sis1} provided that the criterion $\kappa_0<1$ is satisfied, where $\kappa_0$ is given by 
\begin{equation}
\label{kappa0}
\kappa_0=-\int^{2 \pi}_{0}\!\!\int^{+\infty}_{-\infty} F'\!\left(e_0(\theta,v)\right)\left(\frac{\ds \int_{\mathcal D_{e_0(\theta,v)}}(\cos \theta-\cos \theta')(e_0(\theta,v)+m_0\cos \theta')^{-1/2}d\theta'}{\ds\int_{\mathcal D_{e_0(\theta,v)}}(e_0(\theta,v)+m_0\cos \theta')^{-1/2}d\theta'} \right)^2d\theta dv,
\end{equation}
with $${\mathcal D_e}=\left\{\theta'\in \TT\,:\,\,m_0\cos \theta'>-e\right\}.$$

In this paper, we explore situations where this criterion is not satisfied, i.e when $\kappa_0>1$. Let us now state our two main results. The first one concerns the linearized  HMF equation given by
\begin{equation}\label{sislin}
\partial_t f=Lf,
\end{equation}
where 
\begin{equation}\label{defL}
Lf:=-v\partial_{\theta} f+\partial_{\theta} \phi_0\partial_v f+\partial_{\theta} \phi_f\partial_{v} f_0.
\end{equation}

\begin{theorem}[Linear instability]
\label{mainthm1}
Let $f_0\in L^1(\TT,\RR)$ be a stationary solution of \eqref{sis1} of the form \eqref{Qexp}, where $F$ is a nonnegative $\mathcal C^1$ function on $\RR$ such that $F'(e_0(\theta,v))$ belongs  to $L^1(\TT,\RR)$. Assume that $\kappa_0>1$,
where $\kappa_0$ is given by \eqref{kappa0}. Then there exists ${\lambda}>0$ and  a non-zero $f \in L^1(\TT\times\RR)$   such that $e^{\lambda t}f$ is a nontrivial growing mode weak solution to the linearized HMF equation \eqref{sislin}.
\end{theorem}
Our second result is the following nonlinear instability theorem.

\begin{theorem}[Nonlinear instability] \label{mainthm2}
Let $f_0$ be a stationary solution of \eqref{sis1} of the form \eqref{Qexp}, where $F$ is a $\mathcal {C} ^{\infty}$ function on $\RR$, such that $F(e)>0$ for $e<e_*$, $F(e)=0$ for $e\geq e_*$, with $e_*<m_0$ and $|F'(e)|\leq C|e_*-e|^{-\alpha}F(e)$ in the neighborhood of $e_*$, for some $\alpha\geq 1$. Assume that $\kappa_0>1$,
where $\kappa_0$ is given by \eqref{kappa0}. Then $f_0$ is nonlinearly unstable in $L^1(\TT\times\Real)$, namely, there exists $\delta_0>0$  such that for any $\delta>0$ there exists a nonnegative solution $f(t)$ of \eqref{sis1} satisfying $\| f(0)-f_0\|_{L^1}\leq \delta$ and
$$ \| f(t_{\delta})-f_0\|_{L^1}\geq \delta_0,$$
with $t_{\delta}=O(|\log \delta|)$ as $\delta \rightarrow 0$.
\end{theorem}
\begin{remark}
Note that in these two theorems we do not assume that the profile $F$ is a decreasing function. Besides, the set of steady states satisfying the assumptions of these theorems is not empty,  as proved in the Appendix (see Lemma \ref{lemapp}).
Note also that the instability of Theorem \ref{mainthm2} is not due to the usual orbital instability. Indeed the functional space of the pertubation can be restricted to the space of even functions in $(\theta,v)$.
\end{remark}
The outline of the paper is as follows: Sections \ref{section2} and \ref{section3} are respectively devoted to the proofs of Theorem \ref{mainthm1} and Theorem \ref{mainthm2}.

\section{A linear instability result: proof of Theorem \ref{mainthm1}}
\label{section2}
The aim of this section is to prove Theorem \ref{mainthm1}. This proof will be done following the framework used by Lin for the study of periodic BGK waves in \cite{LIN1},  which was generalized  to the analysis of instabilities for the 3D Vlasov-Poisson system by Guo and Lin in \cite{LIN2}. We divide this proof into the three Lemmas \ref{lem1}, \ref{lim1G} and \ref{lim2G}, respectively proved in Subsections \ref{fam}, Subsection \ref{lim0} and Subsection \ref{liminfty}.
 
 \bs
A growing mode of \eqref{sislin} is a solution of the form $e^{\lambda t}f$, where $f\in L^1(\TT,\RR)$ is an unstable eigenfunction of $L$, i.e. a nonzero function satisfying $Lf=\lambda f$ in the sense of distributions, with  $\lambda \in \RR_+^*$ and with $L$ defined by \eqref{defL}. Note that the equation $Lf-\lambda f=0$ is invariant by translation: if $f(\theta,v)$ is an eigenfunction, then for all $\theta_0$, $f(\theta+\theta_0,v)$ is also an eigenfunction. Since for all $f\in L^1$ we can find a $\theta_0\in \TT$ such that $\int \rho_f(\theta+\theta_0)\sin \theta d\theta=0$, we can assume that our eigenfunction of $L$ always satisfy  $\int \rho_f\sin \theta d\theta=0$, i.e. $\phi_f=-m\cos \theta$ with $m=\int\rho_f\cos \theta d\theta$.

Let us first define $\left(\Theta(s,\theta,v),V(s,\theta,v)\right)$ as the solution of the characteristics problem
\begin{equation}\label{siscarac}
\left\{ \begin{array}{l}
\ds\frac{d\Theta(s,\theta,v)}{ds}=V(s,\theta,v)\\[2mm]
\ds\frac{dV(s,\theta,v)}{ds}=-\partial_{\theta}\phi_0\left(\Theta(s,\theta,v)\right)
\end{array} \right. 
\end{equation}
with initial data $\Theta(0,\theta,v)=\theta$, $V(0,\theta,v)=v$. When there is no ambiguity, we denote simply $\Theta(s)=\Theta(s,\theta,v)$ and $V(s)=V(s,\theta,v)$. Since $\phi_0(\theta)=-m_0\cos \theta$, the solution $(\Theta, V)$ is globally defined and belongs to $\mathcal C^\infty(\RR\times\TT\times\RR)$. Note that the energy $e_0\left(\Theta(s),V(s)\right)=\frac{V(s)^2}{2}+\phi_0(\Theta(s))$ does not depend on $s$.

We shall reduce the existence of a growing mode of \eqref{sislin} to the existence of a zero of the following function, defined for all $\lambda\in \RR_+^*$:
\begin{align}
G(\lambda) = &1+\int_{0}^{2 \pi}\int_{\Real}F'\left(\frac{v^2}{2}-m_0\cos\theta\right) \cos^2\theta d\theta dv\nonumber\\
&\quad -\int_{0}^{2 \pi}\int_{\Real}F'\left(\frac{v^2}{2}-m_0\cos\theta\right)\left(\int_{-\infty}^{0}\lambda e^{\lambda s}\cos{\Theta}(s,\theta,v)ds\right) \cos \theta d\theta dv.\label{glambda}
\end{align}

\begin{lemma}
\label{lem1}
Let $f_0\in L^1(\TT,\RR)$ be a stationary solution of \eqref{sis1} of the form \eqref{Qexp}, where $F$ is a $\mathcal C^1$ function on $\RR$ such that $F'(e_0(\theta,v))$ belongs  to $L^1(\TT,\RR)$. Then the function $G$ defined by \eqref{glambda} is well-defined and continuous on $\RR_+^*$.
Moreover, there exists a growing mode $e^{\lambda t}f$ solution to \eqref{sislin} associated with the eigenvalue $\lambda>0$ if and only if $G({\lambda})=0$. An unstable eigenfunction $f$ of $L$ is defined by
\begin{equation}\label{flem2}
f(\theta,v)=-F'(e_0(\theta,v))\cos\theta+F'(e_0(\theta,v))\int_{-\infty}^{0}\lambda e^{\lambda s}\cos\left({\Theta}(s,\theta,v)\right)ds.
\end{equation}
\end{lemma}
\begin{lemma} \label{lim1G}
Under the assumptions of Lemma \ref{lem1}, the function $G$ defined by \eqref{glambda} satisfies
\begin{equation}\label{limG}
\lim \limits_{\,\,\lambda\to 0^+}G(\lambda)=1-\kappa_0,
\end{equation}
where $\kappa_0$ is defined by \eqref{kappa0}.
\end{lemma}
\begin{lemma} \label{lim2G}
Under the assumptions of Lemma \ref{lem1}, the function $G$ defined by \eqref{glambda} satisfies
\begin{equation}\label{limGinfty}
\lim \limits_{\,\,\lambda\to +\infty}G(\lambda)=1.
\end{equation}
\end{lemma}

\begin{proof}[Proof of Theorem \ref{mainthm1}]
From these three lemmas, it is clear that if $\kappa_0>1$, we have $\lim \limits_{\,\,\lambda\to 0^+}G(\lambda)<0$ and $\lim \limits_{\,\,\lambda\to +\infty}G(\lambda)>0$ so by continuity of $G$, there exists $\lambda>0$ such that $G(\lambda)=0$. This means that there exists a growing mode to \eqref{sislin} and this proves Theorem \ref{mainthm1}.
\end{proof}

\subsection{First properties of the function $G(\lambda)$: proof of Lemma \ref{lem1}}
\label{fam}

In this subsection, we prove Lemma \ref{lem1}. Let $\lambda\in\RR_+^*$. Since, by assumption, the function $F'(e_0(\theta,v))$ belongs to $L^1(\TT\times\RR)$, and since 
$$\forall (\theta,v)\qquad \left|\int_{-\infty}^{0}\lambda e^{\lambda s}\cos({\Theta}(s))ds\right|\leq \int_{-\infty}^{0}\lambda e^{\lambda s}ds=1,$$
both functions
$$F'(e_0(\theta,v))\cos\theta\quad\mbox{and}\quad F'(e_0(\theta,v))\int_{-\infty}^{0}\lambda e^{\lambda s}\cos\left({\Theta}(s,\theta,v)\right)ds$$
belong to $L^1(\TT\times \RR)$, so the function $f$ defined by \eqref{flem2} also belongs to $L^1(\TT\times \RR)$.
Hence, by integrating with respect to $\cos \theta d\theta dv$, we deduce that $G(\lambda)$ is well-defined by \eqref{glambda}. The continuity of $G$ on $\RR_+^*$ stems from dominated convergence.

\bs
Consider now a (nonzero) growing mode $e^{\lambda t}f$ of \eqref{sislin} associated to an eigenvalue $\lambda>0$. Let us prove that $G(\lambda)=0$.
From $Lf=\lambda f$ and \eqref{siscarac}, we get, in the sense of distributions,
\begin{align*}
\frac{d}{ds}\left(e^{\lambda s}f\left(\Theta(s),V(s)\right)\right)&=e^{\lambda s}\phi'_f\left(\Theta(s)\right)V(s)F'(e_0(\Theta,V)).
\end{align*}
Integrating this equation from $-R$ to $ 0$, we get, for almost all $(\theta,v)$ and all $R$,
$$
f(\theta,v)=e^{-\lambda R}f\left(\Theta(-R),V(-R)\right)+F'(e_0)\int_{-R}^0 e^{\lambda s} \phi'_f\left(\Theta(s)\right)V(s)ds,
$$
where we recall that $e_0(\Theta,V)=e_0(\theta,v)$.
We multiply by a test function $\psi(\theta,v) \in C_0^{\infty}(\TT\times\RR)$ and integrate with respect to $(\theta,v)$,
\begin{eqnarray*}
\int_{0}^{2\pi}\int_{\RR} f(\theta,v)\psi(\theta,v)d\theta dv&=&e^{-\lambda R}\int_{0}^{2\pi}\int_{\RR} f\left(\theta,v\right)\psi\left(\Theta(R),V(R)\right)d\theta dv\\
&&+\int_{-R}^0\int_{0}^{2\pi}\int_{\RR} e^{\lambda s}F'(e_0) \phi'_f\left(\Theta(s)\right)V(s)\psi\left(\theta,v\right)dsd\theta dv. 
\end{eqnarray*}
In the first integral of the right-hand side, we have performed the change of variable $(\theta,v)=\left (\Theta(R,\theta',v'),V(R,\theta',v')\right)$.
In the second integral, we remark that $|\phi'_f\left(\Theta(s)\right)|\leq \|f\|_{L^1}$ and, the support of $\psi$ being compact, $v$ is bounded. Hence, by $\frac{v^2}{2}-m_0\cos\theta=\frac{V^2}{2}-m_0\cos\Theta$, $V(s)$ is bounded. Therefore, by dominated convergence (using that $F'(e_0)\in L^1$), as $R\rightarrow\infty$, we get
$$
\int_{0}^{2\pi}\int_{\RR} f(\theta,v)\psi(\theta,v)d\theta dv=\int_{-\infty}^{0}\int_{0}^{2\pi}\int_{\RR} e^{\lambda s}F'(e_0) \phi'_f\left(\Theta(s)\right)V(s)\psi\left(\theta,v\right)ds d\theta dv
$$
i.e.
\begin{align*}
f(\theta,v)&=F'(e_0)\int_{-\infty}^0 e^{\lambda s} \phi'_f\left(\Theta(s)\right)V(s)ds=F'(e_0)\int_{-\infty}^0 e^{\lambda s}\frac{d}{ds}\left( \phi_f\left(\Theta(s)\right)\right)ds\\
&=F'(e_0)\phi_f\left(\theta\right)-F'(e_0)\int_{-\infty}^{0}\lambda e^{\lambda s} \phi_f({\Theta}(s))ds
\end{align*}
almost everywhere.
Recall that $\phi_f(\theta)=-m\cos \theta$, with $m=\int^{2 \pi}_{0}\rho_f(\theta)\cos\theta d\theta$. Then we can rewrite this expression of $f$ as
$$
f(\theta,v)=-mF'(e_0)\cos\theta-mF'(e_0)\int_{-\infty}^{0}\lambda e^{\lambda s}\cos {\Theta}(s)ds.
$$
Integrating both sides of this equation with respect to $\cos \theta d\theta dv$ we get
\begin{align*}
m=\int_{0}^{2 \pi}\rho_f(\theta)\cos\theta d\theta =&-m\int_{0}^{2 \pi}\int_{\Real}F'(e_0) \cos^2\theta d\theta dv\\
&+m\int_{0}^{2 \pi}\int_{\Real}F'(e_0)\left(\int_{-\infty}^{0}\lambda e^{\lambda s} \cos {\Theta}(s)ds\right)\cos\theta d\theta dv,
\end{align*}
i.e. $mG(\lambda)=0$. It is clear that $m\neq 0$, otherwise $f=0$ a.e.. Finally, we get $G(\lambda)=0$.

\bs

Reciprocally, assume that $G(\lambda)=0$ for some $\lambda>0$. Let $f$ be given by \eqref{flem2}. We have proved above that this function belongs to $L^1(\TT\times\RR)$. Moreover, since $\Theta(s,-\theta,-v)=\Theta(s,\theta,v)$, we have $\int \rho_f\sin \theta d\theta=0$. Multiplying \eqref{flem2} by $\cos \theta$ and integrating with respect to $\theta$ and $v$, and using that $G(\lambda)=0$, we get $\int \rho_f\cos \theta d\theta=1$, so $f$ is not the zero function and we have $\phi_f(\theta)=-\cos \theta$.

We now check that the function $f$ given by \eqref{flem2} is an eigenfunction of $L$ associated with $\lambda$. From \eqref{flem2}, we get
$$
f(\Theta(t),V(t))=-F'(e_0)\phi_f\left(\Theta(t)\right)+F'(e_0)\int_{-\infty}^{0}\lambda e^{\lambda s}\phi_f\left(\Theta(s,\Theta(t),V(t))\right)ds.
$$
Note that $\Theta(s,\Theta(t),V(t))=\Theta\left(s+t,\theta,v\right)$. Therefore
\begin{align*}
f(\Theta(t),V(t))&=-F'(e_0)\phi_f\left(\Theta(t)\right)+ F'(e_0)\int_{-\infty}^{0}\lambda e^{\lambda s}\phi_f\left(\Theta(s+t)\right)ds\\
&=-F'(e_0)\phi_f\left(\Theta(t)\right) + F'(e_0)e^{-\lambda t}\int_{-\infty}^{t}\lambda e^{\lambda s} \phi_f\left(\Theta(s)\right)ds\\
&=F'(e_0)e^{-\lambda t}\int_{-\infty}^{t} e^{\lambda s}\phi_f'\left(\Theta(s)\right)V(s)ds,
\end{align*}
where we integrated by parts.
Then
\begin{align*}
e^{\lambda t}f(\Theta(t),V(t))&=F'(e_0)\int_{-\infty}^{t} e^{\lambda s}\phi_f'\left(\Theta(s)\right)V(s)ds\\
&=\int_{-\infty}^{t} e^{\lambda s}\phi_f'\left(\Theta(s)\right)\pa_vf_0(\Theta(s),V(s))ds.
\end{align*}
Differentiating both sides with respect to $t$, we obtain in the sense of distributions that for all $t\in \RR$,
\begin{align*}
&&e^{\lambda t}\left(\lambda f(\Theta(t),V(t))+V(t)\pa_\theta f(\Theta(t),V(t))-\phi_0'(\Theta(t))\pa_vf(\Theta(t),V(t))\right)\qquad \\
&&=e^{\lambda t}\phi_f'\left(\Theta(t)\right)\pa_vf_0(\Theta(t),V(t)).
\end{align*}
By writing this equation at $t=0$, we get $Lf=\lambda f$: $f$ is an unstable eigenfunction of $L$. This ends the proof of Lemma \ref{lem1}.
\qed

\subsection{Limiting behavior of $G(\lambda)$ near $\lambda=0$: proof of Lemma \ref{lim1G}}
\label{lim0}

To study  $\lim \limits_{\,\,\lambda\to 0^+}G(\lambda)$ we need to analyze the limit of the function
\begin{equation}\label{littleg}
g_\lambda(\theta,v)=\int_{-\infty}^{0}\lambda e^{\lambda s}\cos{\Theta}(s,\theta,v)ds
\end{equation}
as $\lambda\to 0$. We provide this result in the next lemma, where we also recall some well-known facts on the solution of the characteristics equations \eqref{siscarac}, which are nothing but the pendulum equations.

\begin{lemma}\label{lemperiod}
Let $(\theta,v)\in \TT\times \RR$ and $e_0=\frac{v^2}{2}-m_0\cos\theta$. Consider the solution $\left(\Theta(s,\theta,v),V(s,\theta,v)\right)$ to the characteristics equations \eqref{siscarac}. Then the following holds true.
\begin{itemize}
\item[(i)]
If $e_0>m_0$ then, for all $s\in\RR$, we have
\begin{equation}
\Theta(s+T_{e_0})=\Theta(s)+2\pi,\qquad V(s+T_{e_0})=V(s),
\label{qperiod}
\end{equation}
with
\begin{equation}\label{T1}
T_{e_0}=\int_{0}^{2\pi}\frac{d\theta^{'}}{\sqrt{2\left(e_0+m_0\cos\theta^{'}\right)}}>0.
\end{equation}
\item[(ii)] If $-m_0< e_0<m_0$ then $\Theta$ and $V$ are periodic with period given by
\begin{equation}\label{T2}
{T_{e_0}}=4\int_{0}^{\theta_{m_0}}\frac{d\theta}{\sqrt{2\left(e_0+m_0\cos\theta\right)}}=\frac{4}{\sqrt{m_0}}\int_0^{\pi/2}\frac{d\theta}{\sqrt{1-\frac{m_0+e_0}{2m_0}\sin^2\theta}}>0,
\end{equation}
where $\theta_{m_0}=\arccos(-\frac{e_0}{m_0})$.
\item[(iii)] We have
\begin{equation}\label{deflim}
\lim \limits_{\,\,\lambda\to 0^+}g_\lambda(\theta,v)=\left\{ \begin{array}{l}
\ds\frac{1}{T_{e_0}}\int_{0}^{2 \pi}\!\!\frac{\cos\theta^{'}}{\sqrt{2(e_0+m_0 \cos\theta^{'})}}d\theta^{'}\qquad \text{if}\; e_0>m_0,\\[5mm]
\ds\frac{4}{T_{e_0}}\int_{0}^{\theta_{m_0}}\!\!\frac{\cos\theta'}{\sqrt{2(e_0+m_0 \cos\theta')}}d\theta'\qquad \text{if}\; -m_0\!<\!e_0\!<\!m_0.
\end{array} \right. 
\end{equation}
\end{itemize}
\end{lemma}
\begin{proof}
\textit{(i)} Let $e_0>m_0$. Without loss of generality, since $\Theta(s,-\theta,-v)=\Theta(s,\theta,v)$ and $V(s,-\theta,-v)=V(s,\theta,v)$, we can only treat the case $v>0$. As we have 
$$\frac{V(s)^2}{2}=e_0+m_0\cos \Theta(s)\geq e_0-m_0>0,$$ 
$V(s)$ does not vanish and remains positive. Hence $\Theta(s)$ is the solution of the following autonomous equation
\begin{equation}
\label{eqautonome}\dot\Theta(s)=V(s)=\sqrt{2(e_0+m_0\cos \Theta(s))}
\end{equation}
and is strictly increasing with $\Theta(s)\to +\infty$ as $s\to+\infty$. Let $T_{e_0}$ be the unique time such that $\Theta(T_{e_0})=\theta+2\pi$. By Cauchy-Lipschitz's theorem, we have  $\Theta(s+T_{e_0})-2\pi=\Theta(s)$ and \eqref{qperiod} holds. Defining
$$P(\tau) =\int_0^\tau \frac{d\theta'}{\sqrt{2(e_0+m_0\cos\theta')}},$$
the solution of \eqref{eqautonome} satisfies $P(\Theta(s))-P(\theta)=s$. Therefore, we have $T_{e_0}=P(\theta+2\pi)-P(\theta)$, from which we get \eqref{T1}.

\ms
\ni
\textit{(ii)}  Let $-m_0<e_0<m_0$. In this case, $\Theta(s)$ will oscillate between the two values $\theta_{m_0}=\arccos(-\frac{e_0}{m_0})$ and $-\theta_{m_0}$ with a period $T_{e_0}$ given by \eqref{T2}. On the half-periods where $\Theta$ is increasing, we also have \eqref{eqautonome}. We skip the details of the proof, which is classical.

\ms
\ni
\textit{(iii)}  We remark that $\cos \Theta(s+kT_{e_0})=\cos \Theta(s)$ for all $s\in\RR$ and $k\in \ZZ$. Indeed, by {\em (i)}, for $e_0>m_0$ we have $\Theta(s+kT_{e_0})=\Theta(s)+2\pi k$ and, by {\em (ii)}, for $-m_0<e_0<m_0$ we have $\Theta(s+kT_{e_0})=\Theta(s)$. Hence, we compute from \eqref{littleg}
\begin{align*}
g_\lambda(\theta,v)&=\sum\limits_{k=0}^{+\infty}\int_{kT_{e_0}}^{(k+1)T_{e_0}}\lambda e^{-\lambda s}\cos\Theta(-s)ds \\
&=\sum\limits_{k=0}^{+\infty}\int_{0}^{T_{e_0}}\lambda e^{-\lambda s-k\lambda T_{e_0}}\cos\Theta(-s)ds \\
&=\left(\sum\limits_{k=0}^{+\infty}e^{-k\lambda T_{e_0}}\right)\int_{0}^{T_{e_0}}\lambda e^{-\lambda s}\cos\Theta(-s)ds\\
&=\frac{\lambda}{1-e^{-\lambda T_{e_0}}}\int_{0}^{T_{e_0}}e^{-\lambda s}\cos\Theta(-s)ds.
\end{align*}
Therefore, clearly,
$$
\lim \limits_{\,\,\lambda\to 0^+}g_\lambda(\theta,v)
=\frac{1}{T_{e_0}}\int_{0}^{T_{e_0}}\cos \Theta(-s)ds
=\frac{1}{T_{e_0}}\int_{0}^{T_{e_0}}\cos\Theta (s)ds.
$$
If $e_0>m_0$, we perform the change of variable $\theta'=\Theta(s)$ which is strictly increasing from $[0,T_{e_0}]$ to $[\theta,\theta+2\pi]$. Using \eqref{eqautonome}, we obtain
$$
\lim \limits_{\,\,\lambda\to 0^+}g_\lambda(\theta,v)
=\frac{1}{T_{e_0}}\int_{\theta}^{\theta+2\pi}\frac{\cos \theta'}{\sqrt{2(e_0+m_0\cos \theta')}}d\theta'=\frac{1}{T_{e_0}}\int_{0}^{2\pi}\frac{\cos \theta'}{\sqrt{2(e_0+m_0\cos \theta')}}d\theta'.
$$
If $-m_0<e_0<m_0$, we can always choose a time $t_0$ such that $\Theta(t_0)=-\theta_{m_0}$, $\Theta(t_0+T_{e_0}/2)=\theta_{m_0}$, $\Theta(s)$ is strictly increasing on $[t_0,t_0+T_{e_0}/2]$ and such that $\Theta(s)=\Theta(2t_0+T_{e_0}-s)$ for $s\in [t_0+T_{e_0}/2,t_0+T_{e_0}]$. We have
\begin{align*}
\lim \limits_{\,\,\lambda\to 0^+}g_\lambda(\theta,v)&=\frac{1}{T_{e_0}}\int_{t_0}^{t_0+T_{e_0}/2}\cos\Theta (s)ds+\frac{1}{T_{e_0}}\int_{t_0+T_{e_0}/2}^{t_0+T_{e_0}}\cos\Theta (s)ds\\
&=\frac{1}{T_{e_0}}\int_{t_0}^{t_0+T_{e_0}/2}\cos\Theta (s)ds+\frac{1}{T_{e_0}}\int_{t_0+T_{e_0}/2}^{t_0+T_{e_0}}\cos\Theta (2t_0+T_{e_0}-s)ds\\
&=\frac{2}{T_{e_0}}\int_{t_0}^{t_0+T_{e_0}/2}\cos\Theta (s)ds\\
&=\frac{2}{T_{e_0}}\int_{-\theta_{m_0}}^{\theta_{m_0}}\frac{\cos \theta'}{\sqrt{2(e_0+m_0\cos \theta')}}d\theta'=\frac{4}{T_{e_0}}\int_{0}^{\theta_{m_0}}\frac{\cos \theta'}{\sqrt{2(e_0+m_0\cos \theta')}}d\theta',
\end{align*}
where, on the time interval $[t_0,t_0+T_{e_0}/2]$, we performed the change of variable $\theta'=\Theta(s)$.
\end{proof}

\begin{proof}[Proof of Lemma \ref{lim1G}]
Now we come back to the definition \eqref{glambda} of $G_\lambda$, which reads
\begin{align}
G(\lambda) = &1+\int_{0}^{2 \pi}\int_{\Real}F'\left(\frac{v^2}{2}-m_0\cos\theta\right) \cos^2\theta d\theta dv\nonumber\\
&\quad -\int_{0}^{2 \pi}\int_{\Real}F'\left(\frac{v^2}{2}-m_0\cos\theta\right)g_\lambda(\theta,v) \cos \theta d\theta dv.\label{Glambda2}
\end{align}
We remark that $|g_\lambda(\theta,v)|\leq 1$ and recall that the function $F'\left(\frac{v^2}{2}-m_0\cos\theta\right)$ belongs to $L^1(\TT\times\RR)$. Therefore, we can pass to the limit in the second integral by dominated convergence and deduce from Lemma \ref{lemperiod} {\em (iii)} (note that the set $\{(\theta,v):\,e_0(\theta,v)\leq -m_0\}$ is of measure zero) that

\begin{align*}
\lim \limits_{\,\,\lambda\to 0^+}G(\lambda) &= 1+\int_{0}^{2 \pi}\int_{\Real}F'(e_0) \cos^2\theta d\theta dv\\
&\quad \quad -\iint_{e_0(\theta,v)>m_0}F'(e_0)\left(\frac{1}{T_e}\int_{0}^{2 \pi}\!\!\frac{\cos\theta^{'}}{\sqrt{2(e_0+m_0 \cos\theta^{'})}}d\theta^{'}\right)\cos \theta d\theta dv\\
&\quad \quad -\iint_{-m_0<e_0(\theta,v)<m_0}F'(e_0)\left(\frac{4}{T_e}\int_{0}^{\theta_{m_0}}\!\!\frac{\cos\theta^{'}}{\sqrt{2(e_0+m_0 \cos\theta^{'})}}d\theta^{'}\right)\cos \theta d\theta dv\\
= &1+\int_{0}^{2 \pi}\int_{\Real}F'(e_0)\cos^2\theta d\theta dv\\
&\quad -\int_{0}^{2 \pi}\int_{\Real} F'(e_0)\left(\frac{\ds \int_{\mathcal D_{e_0}}\cos \theta'(e_0+m_0\cos \theta')^{-1/2}d\theta'}{\ds\int_{\mathcal D_{e_0}}(e_0+m_0\cos \theta')^{-1/2}d\theta'} \right)\cos \theta d\theta dv\\
=& 1+\int_{0}^{2 \pi}\int_{\Real}F'(e_0) \cos(\theta)^2dv d\theta-\int_{0}^{2 \pi}\int_{\Real}F'(e_0) \left(\Pi_{m_0}\left(\cos \theta\right)\right)^2 d\theta dv,
\end{align*}
with
\begin{equation}\label{projec}
(\Pi_{m_0} h)(e)=\frac{\ds\int_{\mathcal D_e}(e+m_0\cos \theta)^{-1/2}h(\theta)d\theta}{\ds\int_{\mathcal D_e}(e+m_0\cos \theta)^{-1/2}d\theta},
\end{equation}
for all function $h(\theta)$ and
$${\mathcal D_e}=\left\{\theta'\in \TT\,:\,\,m_0\cos \theta'>-e\right\}.$$
Here the operator $\Pi_{m_0}$ is a variant of the operator $\Pi$  given by (3.8) in \cite{ORB}, this operator should be understood as the ``projector''onto the functions which depend only on the microscopic energy $e_0(\theta,v)$. A projector of this type is also mentioned in the work by Guo and Lin \cite{LIN2}.

\color{black}
Now we remark that straightforward calculations give

\begin{eqnarray}
1-\kappa_0&=&1+\int_{0}^{2 \pi}\int_{\Real}F'(e_0)\left(\cos^2 \theta -2\cos\theta\Pi_{m_0}\left(\cos \theta\right)+\left(\Pi_{m_0}\left(\cos \theta\right)\right)^2\right)\,d\theta dv,\nonumber\\
&=&1+\int_{0}^{2 \pi}\int_{\Real}F'(e_0)\cos^2\theta d\theta dv-\iint F'(e_0) \left(\Pi_{m_0}\left(\cos \theta\right)\right)^2\,d\theta dv,\nonumber
\end{eqnarray}
where
\begin{align*}
-\kappa_0:&=\int_{0}^{2 \pi}\int_{\Real}F'(e_0)\left(\frac{\ds \int_{\mathcal D_{e}}(\cos \theta-\cos \theta')(e_0+m_0\cos \theta')^{-1/2}d\theta'}{\ds\int_{\mathcal D_{e}}(e_0+m_0\cos \theta')^{-1/2}d\theta'} \right)^2d\theta dv.
\end{align*}
This calculation uses that $\Pi_{m_0}$ is a projector. We finally get  \eqref{limG} and the proof of Lemma \ref{lim1G} is complete.
\end{proof}

\subsection{Limiting behavior of $G(\lambda)$ as $\lambda\to\infty$: proof of Lemma \ref{lim2G}}
\label{liminfty}
In this subsection, we prove Lemma \ref{lim2G}. An integration by parts in \eqref{littleg} yields
$$g_\lambda(\theta,v)=\cos\theta+\int_{-\infty}^{0}e^{\lambda s}V(s,\theta,v)\sin{\Theta}(s,\theta,v)ds.$$
The velocity can be bounded independently of $s$ thanks to the conservation of the energy,
$$|V(s)|=\left(v^2+2m_0\cos \theta -2m_0\cos \Theta(s)\right)^{1/2}\leq \left(v^2+4m_0\right)^{1/2}.$$
Thus
$$\left|\int_{-\infty}^{0}e^{\lambda s}\sin{\Theta}(s,\theta,v)V(s,\theta,v)ds\right|\leq \left(v^2+4m_0\right)^{1/2}\int_{-\infty}^{0}e^{\lambda s}ds= \frac{\left(v^2+4m_0\right)^{1/2}}{\lambda}$$
and, for all $(\theta,v)$,
$$\lim_{\lambda\to +\infty}g_\lambda(\theta,v)=\cos \theta.$$
Using again that $|g_\lambda(\theta,v)|\leq 1$ and that $F'(e_0(\theta,v))$ belongs to $L^1$, we deduce directly \eqref{limGinfty} from \eqref{Glambda2} and from dominated convergence. The proof of Lemma \ref{lim2G} is complete.
\qed

\section{A nonlinear instability result: proof of Theorem \ref{mainthm2}}
\label{section3}
We start by an analysis of the linearized HMF operator $L$ around the inhomogeneous equilibrium state $f_0$,  where $L$ is given by \eqref{defL}.
We write
\begin{equation}\label{defL2}
L=L_0  +K, 
\end{equation}
where
\begin{equation}\label{defL0}
L_0f= -v\partial_{\theta} f-E_{f_0}\partial_v f, \qquad Kf=-E_f\partial_{v} f_0, \qquad E_f = -\pa_\theta \phi_f.
\end{equation}

\subsection{Estimates on the semigroup $e^{tL}$}
\label{subsecteL}
Let us state some useful properties of the operator $L_0$ given by \eqref{defL0}.  Since $\phi_0(\theta)=-m_0 \cos\theta$ is smooth, the characteristics equations \eqref{siscarac}  admit a unique  solution $\Theta(s,\theta, v),V(s, \theta, v)$, which is globally defined and $\mathcal C^\infty$ in the variables $(s,\theta,v)$. Moreover, this solution has  bounded derivatives with respect to $\theta$ and $v$, locally in time.   Let $k\in \NN$. For any $f$ in the Sobolev space $W^{k,1}(\TT\times\RR)$, the function 
\begin{equation}
\label{solutionL0}
e^{tL_0} f (t,\theta,v):= f\left( \Theta(-t,\theta,v), V(-t,\theta,v) \right), \quad \forall t\geq 0,
\end{equation}
belongs to $\mathcal C^0(\RR, W^{k,1}(\TT\times\RR))$ and is clearly a solution to $\pa_tg= L_0g$ with initial data $f$.  This means that the semigroup  $e^{tL_0}$ generated by the operator
$L_0$  is strongly continuous on $W^{k,1}(\TT\times\RR)$.

Our aim is to apply the abstract results in \cite{shizuta} concerning perturbation theory of linear operators. Hence, we need to prove the following estimate. For all $\beta>0$, there exists a positive constant $M_{\beta,k}$ such that 
\begin{equation} 
\label{Lestimate}\| e^{tL_0} f\|_{W^{k,1}} \leq M_{\beta,k} e^{t \beta }\| f\|_{W^{k,1}}\qquad \forall f\in W^{k,1}(\TT\times\RR), \quad \forall t\geq 0,
\end{equation}
where $M_{\beta,k}$ depends on $\beta$ and $k$. In fact this estimate will be proved for a subclass of functions $f$.  From the assumptions of Theorem \ref{mainthm2} on $F$, there exists $e_*<m_0$ such that support  of $F$ is $(-\infty,e_*]$. This means that the support of $f_0$ is contained in $\overline{\Omega_0}$, where $\Omega_0$ is the smooth open set
 $$\Omega_0= \left\{(\theta, v): \ \frac{v^2}{2} -m_0\cos\theta < e_*\right\}.$$
We then introduce the functional space
$$\mathscr E_k = \left\{ f \in W^{k,1}(\TT\times\RR): \ \ \mbox{Supp}(f) \subset \overline{\Omega_0} \right\},$$
and claim that for all $f\in \mathscr E_k$, we have $e^{tL_0} f\in \mathscr E_k$ and, for all $\beta>0$,
\begin{equation} 
\label{Lestimatebis} \| e^{tL_0} f\|_{W^{k,1}} \leq M_{\beta,k} e^{t \beta } \| f\|_{W^{k,1}}, \qquad \forall f\in \mathscr E_k,\quad  \forall t\geq 0,
\end{equation}
 $M_{\beta,k}$ being a positive constant depending on $\beta$ and $k$. 

Assume for the moment that estimate \eqref{Lestimatebis} holds true. From the assumptions of Theorem \ref{mainthm2}, one deduces that $\pa_vf_0 \in \mathscr E_k $. It is then easy to check that $e^{tL_0} K$ is a compact operator on
$\mathscr E_k$, for all $t\in \RR$, and the map $t\mapsto e^{tL_0} K \in \mathscr{L}(\mathscr E_k)$  is continuous on $\RR$. Hence, $K$ is $L_0$-smoothing in the sense of \cite{shizuta} (page 707). Assumptions of Theorem 1.1 in \cite{shizuta} are therefore satisfied, which implies that $L$ generates a strongly continuous semigroup $e^{tL}$. Now, from Theorem 1.2 in \cite{shizuta}, for all $\beta >0$, any point of the spectrum $\sigma(L)$ lying in the half plane $\RE z >\beta$ is an isolated eigenvalue  with finite algebraic multiplicity. Furthermore, the set $\sigma (L) \cap \{\RE z >\beta\}$ is finite. 

The assumptions of Theorem \ref{mainthm2} clearly imply those of Theorem \ref{mainthm1}. Hence $L$ admits at least one eigenvalue $\lambda\in \RR_+^*$ associated with an eigenfunction $\widetilde f\in L^1(\TT\times\RR)$. We claim that, in fact, $\widetilde f\in \mathscr E_k$ which will be proved below. This means that the set of eigenvalues of $L$ on $\mathscr E_k$ with positive real part is not empty, and we therefore can choose an eigenvalue $\gamma$ with positive maximal real part. Finally, we apply Theorem 1.3 in \cite{shizuta} and get that, for all $\beta>\RE\gamma$, there exists a positive constant $M_{\beta,k}$ such that
\begin{equation}
\label{Lestimate}\| e^{tL} f\|_{W^{k,1}} \leq M_{\beta,k} \ e^{t \beta }\| f\|_{W^{k,1}}\qquad \forall f\in \mathscr E_k, \quad  \forall t\geq 0.
\end{equation}

\ms

\noindent {\em Proof of \eqref{Lestimatebis} and of the claim $\widetilde f\in \mathscr E_k$}.   Let $f\in \mathscr E_k$. From \eqref{solutionL0}, we clearly have
$$ \|e^{tL_0} f\|_{L^1} = \| f\|_{L^1}, \quad \forall f \in L^1, \quad \forall t\geq 0.$$
Moreover, we know from the analysis of the characteristics problem  \eqref{siscarac} performed in Section \ref{section2}, that by conservation of the energy,  for all $(\theta, v) \in \Omega_0$, we have  $(\Theta(t, \theta,v), V(t, \theta,v)) \in \Omega_0$.
Thus 
$$\mbox{Supp}(e^{tL_0}f) \subset \overline{\Omega_0}.$$
Let $k\geq 1$.
By \eqref{solutionL0}, to get an estimate of $e^{tL_0} f$ in $W^{k,1}(\TT\times\RR)$, it is sufficient to estimate $\Theta$ and $V$ in $W^{k,\infty}(\TT\times\RR)$ for $(\theta,v) \in \Omega_0$. Recall that, since $e_*<m_0$, $\Theta$ and $V$ are periodic functions with period $T_{e_0}$. Moreover, \eqref{T2} shows that $T_{e_0}$ is a $\mathcal C^\infty$ function of $e_0$ on $[-m_0,e_*]$, which means that it is also a $\mathcal C^\infty$ function of $(\theta,v)$. Note also that $\frac{2\pi}{\sqrt{m_0}}\leq T_{e_0}\leq T_{e_*}$.
Define now the following $1$-periodic functions
$$\widetilde \Theta(s,\theta,v)=\Theta\left(sT_{e_0},\theta,v\right),\qquad \widetilde V(s,\theta,v)=V\left(sT_{e_0},\theta,v\right)$$
satisfying
$$\frac{d\widetilde \Theta}{ds}=T_{e_0}\widetilde V,\qquad \frac{d\widetilde V}{ds}=-m_0T_{e_0}\sin\widetilde \Theta.$$
Applying Gronwall Lemma, one gets
$$ |\pa_s^r\pa^j_\theta \pa^\ell_v \widetilde \Theta |+ |\pa_s^r\pa^j_\theta \pa^\ell_v \widetilde V| \leq C_k e^{C_k s },  \quad \forall s\geq 0, \quad \forall (\theta,v)\in \Omega_0, \quad \mbox{for} \ r+j+\ell \leq k.$$
The period of $\widetilde \Theta$ and $\widetilde V$ being independent of $(\theta,v)$, the functions $\pa^j_\theta \pa^\ell_v \widetilde \Theta$ and $\pa^j_\theta \pa^\ell_v \widetilde V$ are also 1-periodic and therefore
$$ |\pa_s^r\pa^j_\theta \pa^\ell_v \widetilde \Theta |+ |\pa_s^r\pa^j_\theta \pa^\ell_v \widetilde V| \leq C_k'=C_ke^{C_k},  \quad \forall s\geq 0, \quad \forall (\theta,v)\in \Omega_0, \quad \mbox{for} \ r+j+\ell \leq k.$$
Coming back to $\Theta$ and $V$, we deduce
\begin{equation}\label{estderiv} |\pa^j_\theta \pa^\ell_v\Theta |+ |\pa^j_\theta \pa^\ell_v  V| \leq C_k(1+s^k),  \quad \forall s\geq 0, \quad \forall (\theta,v)\in \Omega_0, \quad \mbox{for} \ j+\ell \leq k.
\end{equation}
Using this estimate and \eqref{solutionL0}, we finally get \eqref{Lestimatebis}.

Let us finally prove that $\widetilde f\in \mathscr E_k$. By Lemma \ref{lem1}, the function $\widetilde f$ is given by 
$$\widetilde f(\theta,v)=-F'(e_0(\theta,v))\cos\theta+F'(e_0(\theta,v))\int_{-\infty}^{0}\lambda e^{\lambda s}\cos\left({\Theta}(s,\theta,v)\right)ds.$$ Hence, the support of $F'(e_0(\theta,v))$ being in $\overline{\Omega_0}$, the support on $\widetilde f$ will also be contained in $\overline{\Omega_0}$. Moreover, by using \eqref{estderiv}, we obtain that, for some $C_k>0$, we have
$$\forall j+\ell \leq k,\quad \forall (\theta,v)\in \TT\times \RR,\quad\left|\pa^j_\theta \pa^\ell_v\int_{-\infty}^{0}\lambda e^{\lambda s}\cos\left({\Theta}(s,\theta,v)\right)ds\right|\leq C_k.$$
This is sufficient to deduce from $F\in \mathcal C^\infty$ that $\widetilde f\in \mathscr E_k$.

 \subsection{An iterative scheme}

In this part, we prove Theorem \ref{mainthm2} by following the strategy developed by Grenier in \cite{grenier}, which has been also used in \cite{Han-Hauray, Han-Nguyen}
to analyse instabilities for homogeneous steady states of  Vlasov-Poisson models. Let $N\geq 1$ be an integer to be fixed later. According to the previous subsection, we can consider an eigenvalue $\gamma$ of $L$ on $\mathscr E_N$ with maximal real part, $\RE \gamma >0$. Let $g\in \mathscr E_N$ be an associated eigenfunction. With no loss of generality, we may assume that $\|\RE g\|_{L^1}=1$. Let
\begin{equation}
\label{f1}
 f_1(t,\theta,v)= \RE \left( e^{\gamma t} g(\theta,v)\right)\chi_\delta(e_0(\theta,v)),
 \end{equation}
 with $e_0(\theta,v)=\frac{v^2}{2}+\phi_0(\theta)$ and where  $0\leq \chi_\delta(e) \leq 1$ is a smooth real-valued truncation function to be defined further, in order to ensure the positivity of $f_0+\delta f_1(0)$. Note that $f_1$ is almost a growing mode solution to the linearized HMF model \eqref{sislin} since we have
 $$(\partial_t-L)f_1=\RE(e^{\gamma t}\widetilde R_\delta),$$
 where
 \begin{equation}
 \label{Rdelta}
 \widetilde R_\delta=\left(-E_{(1-\chi_\delta)g}+(1-\chi_\delta)E_g\right)\pa_vf_0
 \end{equation}
 will be small.
 We now construct an approximate solution $f^{N}_{app}$ to the HMF model (\ref{sis1}) of the form
\begin{equation}
f^{N}_{app}=f_0+\sum_{k=1}^{N}\delta^{k}f_k,
\end{equation}
for sufficiently small $\delta>0$, in which $f_k$ ($k\geq2$) solves inductively the linear problem
\begin{equation}\label{eqfk}
(\partial_t-L)f_k+\sum_ {j=1}^{k-1}E_{f_j}\partial_v f_{k-j}=0
\end{equation}
with $f_k(0)=0$. Then $f^{N}_{app}$ approximately solves the HMF model (\ref{sis1}) in the sense that
\begin{equation}
\partial_t f^{N}_{app}+v\partial_{\theta} f^{N}_{app}-\partial_{\theta} \phi_{f^{N}_{app}}\partial_v f^{N}_{app}=R_{N}+\delta \RE(e^{\gamma t}\widetilde R),
\end{equation}
where the remainder term $R_{N}$ is given by
\begin{equation}
R_{N}=\sum_{1\leq j,\ell \leq N;j+\ell\geq N+1}\delta^{j+\ell}E_{f_j}\partial_v f_{\ell}.
\end{equation}

\ms

\noindent {\em Step 1. Estimate of $f_k$.} We claim that $f_k \in \mathscr E_{N-k+1}$ and, for all $1\leq k\leq N$,
 \begin{equation}
 \label{estimatefk}
 \|f_k\|_{W^{N-k+1,1}} \leq C_k e^{kt\RE\gamma}.
 \end{equation}
We proceed by induction. From \eqref{f1}, this estimate is a consequence, for $k=1$, of
\begin{equation}
\label{condchi0}
\|g\chi_\delta\|_{W^{N,1}}\leq C_1,
\end{equation}
which is proved below in Step 5. Let $k\geq 2$. We have $$f_k(t)=-\int_{0}^{t}e^{L(t-s)}\sum_ {j=1}^{k-1}E_{f_j}(s)\partial_v f_{k-j}(s)ds.$$
Therefore, for $\RE \gamma < \beta < 2 \RE \gamma$, 
\begin{align*}
\|f_k\|_{W^{N-k+1,1}} & \leq  \sum_ {j=1}^{k-1} \int_0^t \left\|e^{L(t-s)}\left( E_{f_j}(s)\partial_v f_{k-j}(s)\right)\right\|_{W^{N-k+1,1}} ds  \\
   &\leq  M_{\beta, N-k+1} \sum_ {j=1}^{k-1} \int_0^t e^{\beta(t-s)} \left\|E_{f_j}(s)\right\|_{W^{N-k+1,\infty}} \left\|\partial_v f_{k-j}(s)\right\|_{W^{N-k+1,1}} ds \\
   &\leq  M_{\beta, N-k+1} \sum_ {j=1}^{k-1} \int_0^t e^{\beta(t-s)} \left\|f_j(s)\right\|_{L^{1}} \left\|f_{k-j}(s)\right\|_{W^{N-k+2,1}} ds\\
   &\qquad \mbox{since }k-j\leq k-1,\\
   &\leq  M_{\beta, N-k+1} \left(\sum_ {j=1}^{k-1} C_j C_{k-j} \right)\int_0^t e^{\beta(t-s)} e^{ks \RE \gamma } ds\\
   &\leq  \frac{M_{\beta, N-k+1}}{k\RE \gamma -\beta} \left(\sum_ {j=1}^{k-1} C_j C_{k-j} \right)  e^{k t\RE \gamma},
\end{align*}
where we used \eqref{Lestimate} and the recursive assumption. This ends the proof of \eqref{estimatefk}.

\ms 

\noindent {\em Step 2. Estimates of $f_{app}^N-f_0$ and $R_N$.}  The parameter $\delta$ and the time $t$ will be such that 
\begin{equation}
\label{cond0}\delta e^{t\RE \gamma}\leq \min\left(\frac 12,\frac{1}{2K_N}\right),\qquad K_N=\max_{1\leq k\leq N}C_k.
\end{equation}
Hence, from \eqref{estimatefk} we obtain
$$\|f_{app}^N-f_0\|_{W^{1,1}}\leq \sum_{k=1}^N\delta^k C_k e^{kt\RE \gamma}\leq K_N \frac{\delta e^{t\RE\gamma}}{1-\delta e^{t\RE\gamma}}\leq 1 $$
and
$$\|R_N\|_{L^1}\leq \sum_{k=N+1}^{+\infty}\delta^k e^{kt\RE \gamma}\sum_{1\leq j,\ell \leq N;j+\ell =k}C_jC_{\ell}\leq
 \widetilde C_N \left(\delta e^{t\RE \gamma}\right)^{N+1}.$$

\ms 

\noindent {\em Step 3. Estimate of $f-f_{app}^N$.} 
Let $f(t)$ be the solution of \eqref{sis1} with initial data $f_0+\delta \RE g\chi_\delta$ and let $h=f-f_{app}^N$.
Note that the positivity of $f(t)$ is ensured by $f_0+\delta \RE g\chi_\delta\geq 0$ and that we have
$$\|f(0)-f_0\|_{L^1}\leq \delta.$$
The function $h$ satisfies  the following equation
$$\partial_t h+v\partial_{\theta} h+ E_{f}\partial_v h = \left(E_{f_{app}^N} - E_{f}\right) \pa_v f_{app}^N- R_{N}-\delta \RE(e^{\gamma t}\widetilde R_\delta)$$
with $h(0)=0$. To get a $L^1$-estimate of $h$, we multiply this equation by $\mbox{sign}(h)$ and integrate in $(\theta, v)$. We get
\begin{align*}
 \frac{d}{dt} \| h\|_{L^1}& \leq \left\|E_{f_{app}^N} - E_{f}\right\|_{L^\infty} \left\|\pa_v f_{app}^N\right\|_{L^1}+\|R_{N}\|_{L^1}+\delta e^{t\RE \gamma}\|\widetilde R_\delta\|_{L^1}\\
 &\leq  \| h\|_{L^1} \left\|\pa_v f_{app}^N\right\|_{L^1}+\|R_{N}\|_{L^1}+\delta e^{t\RE \gamma}\|\widetilde R_\delta\|_{L^1}.
\end{align*}
From Step 2 we have $\left\|\pa_v f_{app}^N\right\|_{L^1} \leq \left\|\pa_v f_0\right\|_{L^1}+1$, which implies that
$$ \| h(t)\|_{L^1}\leq  \int_0^t   e^{(t-s)(\left\|\pa_v f_0\right\|_{L^1}+1)} \left(\|R_{N}(s)\|_{L^1}+\delta e^{s\RE \gamma}\|\widetilde R_\delta\|_{L^1}\right) ds.$$
Again from Step 2, we then get
  $$ \| h(t)\|_{L^1}\leq \int_0^t   e^{(t-s)(\left\|\pa_v f_0\right\|_{L^1}+1)}  \left( \widetilde C_N\left(\delta e^{s\RE \gamma}\right)^{N+1}+\delta e^{s\RE \gamma}\|\widetilde R_\delta\|_{L^1}\right) ds$$
We now fix $N$ as follows (with the notation $\lfloor \cdot \rfloor$ for the integer function)
$$N:=\left\lfloor\frac{\left\|\pa_v f_0\right\|_{L^1}+1}{\RE \gamma}\right\rfloor+1\geq 1$$
and claim that $\chi$ may be chosen such that
\begin{equation}
\label{condchi}
\|\widetilde R_\delta\|_{L^1}\leq \left(\delta e^{s\RE \gamma}\right)^{N},
\end{equation}
see Step 5 for the proof.
This yields
\begin{equation}
\label{esti1} \| f-f_{app}^N\|_{L^1}(t)\leq \widecheck C_N\left(\delta e^{t\RE \gamma}\right)^{N+1}
\end{equation}
with
$\widecheck C_N=\frac{1+\widetilde C_N}{3\RE \gamma}$.

\ms 

\noindent {\em Step 4. End of the proof.}  Since $\RE g$ is not zero, we can choose a real valued function $\varphi(\theta,v)$ in $L^\infty$ such that
$\|\varphi\|_{L^\infty}$=1 and $$\RE z_g>0 \quad \mbox{with}\quad z_g=\int_0^{2\pi}\int_\RR g \varphi d\theta dv.$$
Denoting
$$z_{g,\delta}=\int_0^{2\pi}\int_\RR g \chi_\delta \varphi d\theta dv,$$
we have
\begin{align*}
\iint f_1\varphi d\theta dv&=e^{t\RE \gamma}\RE\left(e^{it\IM \gamma}z_{g,\delta}\right)\\
&\geq e^{t\RE \gamma}\RE\left(e^{it\IM \gamma}z_g\right)-e^{t\RE \gamma} |z_g-z_{g,\delta}|\\
&\geq e^{t\RE \gamma}\RE\left(e^{it\IM \gamma}z_g\right)-e^{t\RE \gamma} \|g(1-\chi_\delta)\|_{L^1}
\end{align*}
We claim that
\begin{equation}
\label{condchi2}
\lim_{\delta\to 0}\|(1-\chi_\delta)g\|_{L^1}=0,
\end{equation}
which again will be proved in Step 5.
In order to end the proof of Theorem \ref{mainthm2}, we estimate from below, using \eqref{esti1} and \eqref{estimatefk},
\begin{align*}\|f-f_0\|_{L^1}& \geq \iint (f-f_0)\varphi d\theta dv=\iint (f_{app}^N-f_0)\varphi d\theta dv+\iint (f-f_{app}^N)\varphi d\theta dv\\
&\geq \delta \iint f_1\varphi d\theta dv-\sum_{k=2}^N\delta^k\|f_k\|_{L^1} - \widecheck C_N\left(\delta e^{t\RE \gamma}\right)^{N+1}\\
&\geq \delta  \iint f_1\varphi d\theta dv-\sum_{k=2}^NC_k\left(\delta e^{t\RE \gamma}\right)^k - \widecheck C_N\left(\delta e^{t\RE \gamma}\right)^{N+1}\\
&\geq \delta \iint f_1\varphi d\theta dv-2K_N\left(\delta e^{t\RE \gamma}\right)^2- \widecheck C_N\left(\delta e^{t\RE \gamma}\right)^{N+1}\\
&\geq\delta e^{t\RE \gamma}\left(\RE\left(e^{it\IM \gamma}z_g\right)- \|(1-\chi_\delta)g\|_{L^1}-2K_N\delta e^{t\RE \gamma} - \widecheck C_N\left(\delta e^{t\RE \gamma}\right)^{N}\right)
\end{align*}
Assume for a while that \begin{equation}\label{condition}\RE\left(e^{it\IM \gamma}z_g\right)\geq \frac{\RE z_g}{2}.\end{equation} We have
$$\|f-f_0\|_{L^1}\geq \frac{\delta e^{t\RE \gamma}\RE z_g}{2}\left(1-\frac{2\|(1-\chi_\delta)g\|_{L^1}}{\RE z_g}-\frac{4K_N\delta e^{t\RE \gamma}}{\RE z_g} - \frac{2\widecheck C_N\left(\delta e^{t\RE \gamma}\right)^{N}}{\RE z_g}\right)$$
Let $\delta_0>0$ be such that
 $$\frac{32K_N}{(\RE z_g)^2}\delta_0 + \frac{2 \widecheck C_N8^N}{(\RE z_g)^{N+1}}\delta_0^{N}\leq \frac14\quad \mbox{and}\quad \frac{8\delta_0 }{\RE z_g}\leq \min \left(\frac 12,\frac{1}{2K_N}\right)$$
(note that $N\geq 1$) and consider times $t$ such that
\begin{equation}\label{tdelta}\frac{4\delta_0 }{\RE z_g}\leq \delta e^{t\RE \gamma}\leq \frac{8\delta_0 }{\RE z_g}.\end{equation}
Owing to \eqref{condchi2}, we also choose $\delta$ small enough such that
$$\frac{2\|(1-\chi_\delta)g\|_{L^1}}{\RE z_g}\leq\frac{1}{4}.$$
We conclude from these inequalities that
$$\|f-f_0\|_{L^1}\geq \delta_0$$
and that \eqref{cond0} is satisfied. 

\bs
To end the proof, it remains to fix the time $t_\delta$ and to choose the truncation function $\chi_\delta$. Let us  show that, for $\delta$ small enough, there exists a time $t_\delta$ satisfying both \eqref{condition} and \eqref{tdelta}.
If $\IM \gamma =0$, then \eqref{condition} is clearly satisfied since $\RE z_g>0$: a suitable $t_\delta$ is then
$$t_\delta=\frac{1}{\RE \gamma}\log \left(\frac{6\delta_0}{\delta \RE z_g}\right).$$
Assume now that $\IM \gamma\neq 0$. For $\delta$ small enough, the size of the interval of times $t$ satisfying \eqref{tdelta} becomes larger than $\frac{2\pi}{|\IM \gamma|}$. This means that it is possible to find a time $t_\delta$ in this interval satisfying \eqref{condition}.

\ms 

\noindent {\em Step 5. Choice of $\chi_\delta$.}  
For all $\delta>0$, we have to fix the function $\chi_\delta\in \mathcal C^\infty(\RR)$ such that \eqref{condchi0}, \eqref{condchi}, \eqref{condchi2} are satisfied and such that $f_0+\delta f_1(0)\geq 0$. First of all, proceeding as in the proof of Lemma \ref{lem1}, we obtain that $g$ takes the form
\begin{equation}
\label{eqqg}
g(\theta,v)=-mF'(e_0)\cos\theta-mF'(e_0)\int_{-\infty}^{0}\gamma e^{\gamma s}\cos {\Theta}(s)ds,
\end{equation}
with $m=\int^{2 \pi}_{0}\rho_f(\theta)\cos\theta d\theta$. Hence, 
\begin{equation}
\label{majg}
|g(\theta,v)|\leq |m|\left(1+\frac{|\gamma|}{\RE \gamma}\right)|F'(e_0)|.
\end{equation}
The assumptions on $F$ and $F'$ in Theorem \ref{mainthm2} imply that
\begin{equation}
\label{alpha}\forall e<e_*,\qquad |F'(e)|\leq C(e_*-e)^{-\alpha}F(e)
\end{equation}
with $\alpha\geq 1$. Since $F(e)>0$ for $e<e_*$, the local assumption becomes global. Let $\chi$ be a $\mathcal C^\infty$ function such that $0\leq \chi\leq 1$ and
$$
\left\{\begin{array}{ll}
\ds \chi(t)=0\quad &\mbox{for }t\leq 0,\\
\ds \chi(t)\leq 2t^\alpha\quad &\mbox{for }t\geq 0,\\
\ds \chi(t)=1\quad &\mbox{for }t\geq 1
\end{array}\right.
$$
and let
\begin{equation}
\label{chidelta}\chi_\delta(e)=\chi\left(\frac{e_*-e}{\delta^{1/{(2\alpha)}}}\right).
\end{equation}
From 
$$\| (1-\chi_\delta)g\|_{L^1}\leq \|g\un_{e_*-\delta^{1/(2\alpha)}< e_0(\theta,v)<e_*}\|_{L^1}$$
and dominated convergence, we clearly have \eqref{condchi2}.
By \eqref{majg} and \eqref{alpha}, we have, for all $(\theta,v)\in \TT\times \RR$,
$$\delta\left|\RE g(\theta,v)\chi_\delta(e_0(\theta,v))\right|\leq \delta C|e_*-e_0|^{-\alpha}F(e_0)\frac{|e_*-e_0|^{\alpha}}{\delta^{1/2}}=C\delta^{1/2}f_0(\theta,v),$$
so for $\delta$ small enough, we have $f_0+\delta f_1(0)\geq 0$. 

By differentiating \eqref{eqqg} and using \eqref{estderiv}, we get
$$\forall j+\ell \leq N,\qquad |\pa^j_\theta\pa^\ell_v g(\theta,v)|\leq C\max_{k\leq N+1}F^{(k)}(e_0)\leq C|e_*-e_0|^{N-1},$$
where we used Taylor formulas and the fact that $F\in \mathcal C^\infty$ with $F(e)=0$ for $e\geq e_*$. Besides, from \eqref{chidelta}, we obtain (if $\delta\leq 1$)
$$\forall \ 1\leq j+\ell \leq N,\qquad |\pa^j_\theta\pa^\ell_v \chi_\delta(e_0(\theta,v))|\leq C\delta^{-N/(2\alpha)}\un_{e_*-\delta^{1/(2\alpha)}< e_0(\theta,v)< e_*}.$$
Therefore
\begin{align*}
\|g\chi_\delta\|_{W^{N,1}}&\leq \|g\|_{L^1}+C\delta^{-N/(2\alpha)}\int_0^{2\pi}\int_\RR|e_*-e_0(\theta,v)|^{N-1}\un_{e_*-\delta^{1/(2\alpha)}< e_0(\theta,v)< e_*}d\theta dv\\
&\leq \|g\|_{L^1}+C\delta^{-N/(2\alpha)}\int_{e_*-\delta^{1/(2\alpha)}}^{e_*}(e_*-e)^{N-1} \left(4\int_{0}^{\theta_{m_0}}\frac{d\theta}{\sqrt{2\left(e+m_0\cos\theta\right)}}\right)de\\
&\quad= \|g\|_{L^1}+C\delta^{-N/(2\alpha)}\int_{e_*-\delta^{1/(2\alpha)}}^{e_*}(e_*-e)^{N-1} \,T_e\,de,
\end{align*}
where $\theta_{m_0}=\arccos(-\frac{e}{m_0})$ and $T_e$ is given by \eqref{T2}. Now we recall that for $e\leq e_*$ we have $T_e\leq T_{e_*}$. This yields
$$\|g\chi_\delta\|_{W^{N,1}}\leq  \|g\|_{L^1}+CT_{e_*}\delta^{-N/(2\alpha)}\int_{e_*-\delta^{1/(2\alpha)}}^{e_*}(e_*-e)^{N-1}\,de=\|g\|_{L^1}+\frac{CT_{e_*}}{N}.$$
We have proved \eqref{condchi0}.

By \eqref{Rdelta}, we have 
\begin{align*}
\|\widetilde R_\delta\|_{L^1}&\leq C\left(\| (1-\chi_\delta)g\|_{L^1}+\|(1-\chi_\delta)\pa_vf_0\|_{L^1}\right)\\
&\leq C\left(\|g\un_{e_*-\delta^{1/(2\alpha)}< e_0(\theta,v)<e_*}\|_{L^1}+\|\pa_vf_0\un_{e_*-\delta^{1/(2\alpha)}<e_0(\theta,v)< e_*}\|_{L^1}\right),
\end{align*}
so by dominated convergence,
$$\lim_{\delta\to 0}\|\widetilde R_\delta\|_{L^1}=0.$$
We now choose $\delta$ small enough such that
$$\|R_\delta\|_{L^1}\leq \left(\frac{4\delta_0 }{\RE z_g}\right)^N.$$
From \eqref{tdelta}, we obtain \eqref{condchi}, which ends the proof of Theorem \ref{mainthm2}.
\qed

\appendix
\section{Appendix. Existence of unstable steady states}
\label{A1}
In this section, we prove that the set of steady states satisfying the assumptions of Theorems \ref{mainthm1} and \ref{mainthm2} is not empty.
More precisely, we prove the following
\begin{lemma}
\label{lemapp}
Let $m>0$. There exist $m>0$, $e_*<m$ and there exists a nonincreasing function $F$, $\mathcal {C} ^{\infty}$ on $\RR$, such that $F(e)>0$ for $e<e_*$, $F(e)=0$ for $e\geq e_*$ and $|F'(e)|\leq C|e_*-e|^{-\alpha}F(e)$ in the neighborhood of $e_*$, for some $\alpha\geq 1$, and such that the function $f(\theta,v)=F(\frac{v^2}{2}-m\cos \theta)$ is a steady state solution to the HMF model \eqref{sis1} and such that $\kappa(m,F)>1$,
where $\kappa(m,F)$ is given by
$$
\kappa(m,F)=\int^{2 \pi}_{0}\!\!\int^{+\infty}_{-\infty} \left|F'\!\left(e(\theta,v)\right)\right|\left(\frac{\ds \int_{\mathcal D_{e(\theta,v)}}(\cos \theta-\cos \theta')(e(\theta,v)+m\cos \theta')^{-1/2}d\theta'}{\ds\int_{\mathcal D_{e(\theta,v)}}(e(\theta,v)+m\cos \theta')^{-1/2}d\theta'} \right)^2d\theta dv,
$$
with $$e(\theta,v)=\frac{v^2}{2}-m\cos \theta,\qquad {\mathcal D_e}=\left\{\theta'\in \TT\,:\,\,m\cos\theta'>-e\right\}.$$
\end{lemma}
\begin{proof}
Let $m>0$ and $F$ a nonincreasing $\mathcal {C} ^{\infty}$ function on $\RR$ supported in $(-\infty,m)$, which is not identically zero on $(-m,m)$. We first observe that $f(\theta,v)=F(\frac{v^2}{2}-m\cos \theta)$ is a steady state solution to the HMF model \eqref{sis1} if and only if $m$ and $F$ satisfy $\gamma(m,F)=m$ with
$$\gamma(m,F):=\int_0^{2\pi}\int_\RR F\left(\frac{v^2}{2}-m\cos\theta\right)\cos \theta d\theta dv>0.$$
By using the linearity of $\gamma$ in $F$ we deduce that $\frac{m}{\gamma(m,F)}F(\frac{v^2}{2}-m\cos \theta)$ is a steady state.

We proceed by a contradiction argument. Assume that
$$\kappa\left(m,\frac{m}{\gamma(m,F)}F\right)\leq 1$$
 for all $m>0$ and all nonincreasing $\mathcal {C} ^{\infty}$ function $F$ supported in $(-\infty,m)$ such that, denoting by $(-\infty,e_*]$ the support of $F$, we have $|F'(e)|\leq C|e_*-e|^{-\alpha}F(e)$ in the neighborhood of $e_*$, for some $\alpha\geq 1$. This is equivalent to
 $$\kappa\left(m,F\right)\leq \frac{\gamma(m,F)}{m},$$
or, after straightforward calculation and an integration by parts,
\begin{equation}
\label{num}-\iint F'\left(e(\theta,v)\right)g_m(e(\theta,v))d\theta dv\leq 0
\end{equation}
with
$$g_m(e)=(\Pi_m \cos^2\theta)(e)-\left((\Pi_m\cos \theta)(e)\right)^2-(\Pi_m \sin^2\theta)(e)$$
and for all function $h(\theta)$,
$$(\Pi_m h)(e)=\frac{\ds\int_{\mathcal D_e}(e+m\cos \theta)^{-1/2}h(\theta)d\theta}{\ds\int_{\mathcal D_e}(e+m\cos \theta)^{-1/2}d\theta}.$$
Now, we choose the functions $F$ as follows. We first pick a nonincreasing $\mathcal {C} ^{\infty}$ function $\Psi$ on $\RR$ with support $(\infty,e_\sharp]\subset(-\infty,m)$, then we set $e_*=\frac{e_\sharp+m}{2}$ and define
$$F_\eps(e)=\Psi(e)+\eps \exp\left(-(e_*-e)^{-1}\right), \quad \mbox{for }e<e_*,$$
the parameter $\eps>0$ being arbitrary. Since $F_\eps$ satisfies the assumptions, it satisfies \eqref{num}. Then, letting $\eps\to 0$, we get
$$-\iint \Psi'\left(e(\theta,v)\right)g_m(e(\theta,v))d\theta dv\leq 0.$$
The function $\Psi$ being arbitrary, this is equivalent to 
$$g_m(e)\leq 0,\qquad \forall m>0,\quad \forall e\in (-m,m),$$
or,
\begin{equation}
\label{g1}
g_1(e)\leq 0,\qquad \forall e\in (-1,1).
\end{equation}
Let us now prove that the function $g_1(e)$ is in fact positive in the neighborhood of $e=1$, which contradicts \eqref{g1}.

Indeed, we introduce
$$\alpha(e)=\int_{\mathcal D_e}(e+\cos\theta)^{-1/2}d\theta,\qquad \beta(e)=\int_{\mathcal D_e}(e+\cos\theta)^{-1/2}\sin^2\theta d\theta.$$
We have
\begin{align*}\alpha(e) g_1(e)
&=\alpha(e)-2\beta(e)-\frac{1}{\alpha(e)}\left(\int_{\mathcal D_e}(e+\cos\theta)^{1/2}d\theta-e\alpha(e)\right)^2\\
&=(1-e^2)\alpha(e)-2\beta(e)+2e\int_{\mathcal D_e}(e+\cos\theta)^{1/2}d\theta-\frac{1}{\alpha(e)}\left(\int_{\mathcal D_e}(e+\cos\theta)^{1/2}d\theta\right)^2.
\end{align*}
From \cite{llmehats}, we have
$$\alpha(e)\sim-\sqrt{2}\log(1-e)\quad \mbox{as }e\to 1^-,$$
and direct calculations yield
$$\int_0^{2\pi}(1+\cos\theta)^{1/2}d\theta=4\sqrt{2},\qquad \beta(1)=\frac{8\sqrt{2}}{3}.$$
This means that 
$$\alpha(e)g_1(e)\to \frac{8\sqrt{2}}{3}>0 \quad\mbox{as }e\to 1^-,  $$
$$ g_1(e)\sim \frac{8\sqrt{2}}{\alpha(e)}\mbox{as }e\to 1^-.$$
This proves the claim.
\end{proof}

\subsubsection*{Acknowledgments}
The authors wish to thank D. Han-Kwan for helpful discussions. A. M. Luz acknowledges support by the Brazilian National Council for Scientific and Technological Development (CNPq) under the program ``Science without Borders’’ 249279/2013-4. M. Lemou and F. M\'ehats acknowledge supports from the ANR project  MOONRISE ANR-14-CE23-0007-01, from the ENS Rennes project MUNIQ and from the INRIA project ANTIPODE.

{}

\begin{thebibliography}{}

\bibitem{antoni-ruffo} M. Antoni, S. Ruffo, Clustering and relaxation in Hamiltonian long-range dynamics, Phys. Rev. E, {\bf 52} (1995), 2361.
\bibitem{antoniazzi} A. Antoniazzi, D. Fanelli, S. Ruffo, Y. Y. Yamaguchi, Nonequilibrium tricritical point in a system with long-range interactions, Phys. Rev. Lett. {\bf 99} (2007), 040601.
\bibitem{barre2} J. Barr\'e, F. Bouchet, T. Dauxois, S. Ruffo, Y. Y. Yamaguchi, The Vlasov equation and the Hamiltonian mean-field model, Physica A {\bf 365} (2006), 177.
\bibitem{barre3} J. Barr\'e, A. Olivetti, Y. Y. Yamaguchi, Dynamics of perturbations around inhomogeneous backgrounds in the HMF model, J. Stat. Mech. (2010), 08002.
\bibitem{barre4} J. Barr\'e, A. Olivetti, Y. Y. Yamaguchi, Algebraic damping in the one-dimensional Vlasov equation, J. Phys. A: Math.
Gen. {\bf 44} (2011), 405502.
\bibitem{yamaguchi1} J. Barr\'e, Y. Y. Yamaguchi, Small traveling clusters in attractive and repulsive Hamiltonian mean-field models, Phys. Rev. E {\bf 79} (2009), 036208. 
\bibitem{yamaguchi3} J. Barr\'e, Y. Y. Yamaguchi,  On the neighborhood of an inhomogeneous stable stationary solution of the Vlasov equation -- Case of an attractive cosine potential, J. Math. Phys. 56 (2015), 081502.
\bibitem{rousset1} E. Caglioti, F. Rousset, Long time estimates in the mean field limit, Arch. Ration. Mech. Anal. {\bf 190} (2008), no. 3, 517--547.
\bibitem{rousset2} E. Caglioti, F. Rousset, Quasi-stationary states for particle systems in the mean-field limit, J. Stat. Phys. {\bf 129} (2007), no. 2, 241--263.
\bibitem{chavanis3} A. Campa, P.-H. Chavanis, Inhomogeneous Tsallis distributions in the HMF model, J. Stat. Mech. (2010), 06001.
\bibitem{chavanis1} P.-H. Chavanis, Lynden-Bell and Tsallis distributions for the HMF model, Eur. Phys. J. B {\bf 53} (2006), 487.
\bibitem{chavanis4} P.-H. Chavanis, J. Vatteville,  F. Bouchet. Dynamics and thermodynamics of a simple model similar to self-gravitating systems : the HMF model, Eur. Phys. J. B, {\bf 46} (2005), 61.
\bibitem{faou-rousset} E. Faou, F. Rousset, Landau damping in Sobolev spaces for the Vlasov-HMF model, Arch. Ration. Mech. Anal. {\bf 219} (2016), no. 2, 887--902.
\bibitem{grenier} E. Grenier, On the nonlinear instability of Euler and Prandtl equations, Comm. Pure Appl. Math. {\bf 53} (2000), no. 9, 1067--1091.
\bibitem{LIN2} Y. Guo, Z. Lin, Unstable and stable Galaxy models, Comm. Math. Phys. {\bf 279} (2008), 789--813.
\bibitem{guo-strauss} Y. Guo, W. Strauss, Nonlinear instability of double-humped equilibria, Ann. Inst. H. Poincar\'e Anal. Non Lin\'eaire {\bf 12} (1995), 339--352.
\bibitem{Han-Hauray} D. Han-Kwan, M. Hauray, Stability issues in the quasineutral limit of the one-dimensional Vlasov-Poisson
equation, Comm. Math. Phys. {\bf 334} (2015), no. 2, 1101--1152.
\bibitem{Han-Nguyen} D. Han-Kwan, T. Nguyen,  Instabilities in the mean field limit, J. Stat. Phys. {\bf 162} (2016), no. 6, 1639--1653.
\bibitem{LIN1} Z. Lin, Instability of periodic BGK waves, Math. Res. Letts. {\bf 8} (2001), 521--534.
\bibitem{LMR1} M. Lemou, F. M\'ehats, P. Rapha\"el, Structure of the linearized gravitational Vlasov-Poisson system close to a polytropic ground state, SIAM J. Math. Anal. {\bf 39} (2008), no. 6, 1711--1739.
\bibitem{ORB} M. Lemou, F. M\'ehats, P. Rapha\"el, Orbital stability of spherical galactic models, Inventiones Math. {\bf 187} (2012), 145--194.
\bibitem{llmehats}  M. Lemou, A. M. Luz, F. M\'ehats, Nonlinear stability criteria for the HMF Model, Arch. Rational Mech. Anal. {\bf 224} (2017), no. 2, 353--380.
\bibitem{OGW} S. Ogawa, Spectral and formal stability criteria of spatially inhomogeneous solutions to the Vlasov equation for the Hamiltonian mean-field model, Phys. Rev. E {\bf 87} (2013), 062107.
\bibitem{ogawa2} S. Ogawa and Y. Y. Yamaguchi, Precise determination of the nonequilibrium tricritical point based on Lynden-Bell theory in the Hamiltonian mean-field model, Phys. Rev. E, {\bf 84} (2011), 061140.
\bibitem{shizuta} Y. Shizuta, On the classical solutions of the Boltzmann equation, Comm. Pure Appl. Math. 36 (1983), no. 6, 705--754.
\bibitem{chavanis2} F. Staniscia, P. H. Chavanis, G. De Ninno, Out-of-equilibrium phase transitions in the HMF model : a closer look, Phys. Rev. E. {\bf 83} (2011), 051111.
\bibitem{yamaguchi2} Y. Y. Yamaguchi, Construction of traveling clusters in the Hamiltonian mean-field model by nonequilibrium statistical mechanics and Bernstein-Greene-Kruskal waves, Phys. Rev. E {\bf 84} (2011), 016211.
\bibitem{barre1} Y. Y. Yamaguchi, J. Barr\'e, F. Bouchet, T. Dauxois, S. Ruffo, Stability criteria of the Vlasov equation and quasi-stationary states of the HMF model, Physica A {\bf 337} (2004), 36.
\end{thebibliography}
\end{document}